\documentclass[10pt]{amsart}
\usepackage{amsmath,amsthm,amssymb}

\def\id{{\rm{id}}}

\def\alg#1{\mathbb{#1}}
\def\family#1{\mathcal{#1}}
\def\cat#1{\mathbf{#1}}
\def\ASF{\cat{AccSetFun}}
\def\st{\,;\,}
\def\to{\rightarrow}

\def\modA{/\!\!\approx_{\alg{A}}}
\def\modB{/\!\!\approx_{\alg{B}}}

\newtheorem{theorem}{Theorem}[section]
\newtheorem{definition}[theorem]{Definition}
\newtheorem{lemma}[theorem]{Lemma}

\numberwithin{equation}{section}

\begin{document}
\title[Accessible set functors are universal]{Accessible set functors are universal}
\author{Libor Barto}
\address{Charles University\\Faculty of Mathematics and Physics\\Department of Algebra\\Sokolovsk\'a 83\\186 75 Praha 8\\Czech Republic}
\email{libor.barto@gmail.com}
\thanks{
Libor Barto has received funding from the Czech Science Foundation (grant No 201/06/0664) and 
from the European Research Council
(ERC) under the European Unions Horizon 2020 research and
innovation programme (grant agreement No 771005).
}
\date{March 28, 2019}
\subjclass{Primary 18B15, 18A22, 18A25; Secondary 08B05}
\keywords{set functor, universal category, full embedding}
\dedicatory{Dedicated to the memory of V\v era Trnkov\'a.}

\begin{abstract}
It is shown that every concretizable category can be fully embedded into the category of accessible set functors and natural transformations. 
\end{abstract}

\maketitle

\section{Introduction}

This paper contributes to two areas in which the work of V\v era Trnkov\'a plays a major role: universal categories and set functors. 

We start by providing some background to these areas and then we state the main result of this paper. The discussion is rather informal and some fine technical points are neglected. 
The essential concepts for this paper are formally introduced in Section~\ref{sec:prelim}.

\subsection{Universality}

One approach to measuring the complexity of a category is to characterize its full subcategories.

There is usually a natural limitation on the possible full subcategories. The following two situations are especially common: if the category in question is \emph{concretizable} (i.e., admits a faithful functor to the category of sets), then every full subcategory is concretizable as well; if the category is \emph{algebraic} (i.e., is isomorphic to a full subcategory of a category of universal algebras), then so are all of its full subcategories.

It turned out that such natural limitations are the only ones for many categories of interest~\cite{PT,T5}. An algebraic category is
often \emph{alg--universal}, i.e., each category of universal algebras is isomorphic to its full subcategory. Examples include
the category of graphs%
\footnote{If morphisms are not specified, then the most natural choice is meant. Here the morphisms are the edge--preserving mappings.}
and the category of algebras with two unary operations~\cite{HP66}.
Concretizable categories are sometimes even \emph{universal}, i.e., each concretizable category is among the isomorphic copies of their full subcategories. This is the case, e.g., for the category of hypergraphs~(a combination of results of Hedrl\'in and Ku\v cera, see \cite{PT}),
the category of topological semigroups~\cite{T6}, or the category of topological spaces with open continuous
maps as morphisms~\cite{PT}.

Although it is not immediately seen from the definition, alg--universal categories are indeed comprehensive. 
For instance, they contain an isomorphic copy of every \emph{small} category, i.e., a category whose object class is a set (see \cite{PT}). In particular, every partially ordered set can be represented as a set of objects with the relation ``there exists a morphism'' and, also, every group can be represented as the automorphism group of some object. 
The latter property of a category, \emph{group--universality}, predates the concept of alg--universality and even the notion of a category. It is a remarkable achievement of the theory that the group--universality is now usually best shown by proving the far stronger alg--universality. In fact, the development leading to this paper, discussed in Subsection~\ref{subsec:res}, is an instance of this phenomenon.

Universal categories may seem way more comprehensive than alg--universal ones. For instance, they can represent in the above sense every partially ordered class. However, a combination of results of Ku\v cera, Pultr, and Hedrl\'in proves (see, again, \cite{PT}) that the statement ``every alg--universal category is universal'' is equivalent to the set--theoretical assumption ``the class of all measurable cardinals is a set''. It follows that, in some models of set theory, the alg--universality and universality coincide. Still, an absolute proof of universality indicates, philosophically speaking, that the category is more comprehensive than any algebraic category.

\subsection{Set functors}

\emph{Set functors} are simply functors from the category $\cat{Set}$ of all sets and mappings to itself. The natural choice of morphisms between them is natural transformations. 
Besides being basic examples of functors (which perhaps motivated early results from the 1970s, e.g., \cite{T1,T2,T3,K1}), set functors are also among the central concepts in both Universal Coalgebra and Universal Algebra. Their appearance in the former is direct and explicit as type functors for coalgebras that model various state base systems such as streams, automata, or Kripke structures~\cite{Rut00}. The significance in the latter is a bit less direct and not so widely acknowledged in the Universal Algebra community.

One of the main focuses of Universal Algebra is the study of systems of universally quantified equations over some signature (or classes of all models of such systems --- varieties). 
The special attention to these objects has paid large dividends in theoretical computer science, namely in the computational complexity of constraint satisfaction problems (CSPs, for short)%
\footnote{A recent survey is~\cite{BKW17}. More recently, a solution to the main open problem of the area was announced independently by Bulatov~\cite{Bul17} and Zhuk~\cite{Zhu17}. Both proofs heavily use Universal Algebra.},
where to each CSP one assigns a system of equations which captures the computational complexity of that CSP, and morphisms between these systems 
then correspond to certain natural polynomial--time reductions. 
It has turned out~\cite{BOP18} that, in this context, it is enough to  
consider particularly simple equations, those that contain exactly one function symbol on both sides, i.e., equations of the form $t(\mbox{some variables}) = s(\mbox{some variables})$. The importance of such equations, called \emph{height one equations} in~\cite{BOP18} or \emph{flat equations} in~\cite{AGT10}, is further witnessed by a very recent development in promise CSPs~\cite{BKO}.

Returning back to set functors --- their category is equivalent to the category of systems of height one equations!%
\footnote{
This claim is rather sloppy; for instance, the collection of all set functors is not a class, so set functors do not form a legitimate category.} 
We will now sketch one direction of this equivalence and refer the reader to a nice presentation in~\cite{AGT10} for details.
Given a set functor $F$ we regard members of $FX$ as function symbols of arity $|X|$ and, for each mapping $f: X \to Y$ and $t \in FX$, we add one equation to the system in a natural way, e.g., if 
$X=\{x_1, x_2, x_3\}$,  
$Y = \{y_1,y_2\}$,  
$f(x_1)=f(x_2)=y_2$, $f(x_3)=y_1$,  
$t \in FX$, and 
$s = Ff(t)$, 
then we add the equation $t(y_2,y_2,y_1) = s(y_1,y_2)$.  
Natural transformations of set functors then correspond to mappings preserving height one equations.

Note that the obtained system of equations forms a proper class (except for uninteresting functors with $FX = \emptyset$ for all $X \neq \emptyset$) and involves infinitary function symbols, which is rather non-standard in Universal Algebra. 
However, if we restrict to \emph{accessible} set functors (to be defined in Section~\ref{sec:prelim}), then it is enough to introduce function symbols for elements of $FX$ for a sufficiently large $X$. Restricting further to \emph{finitary} functors allows us to introduce function symbols only for elements of $F(\{x_1, \dots, x_n\})$, $n \in \mathbb N$, and we are back in a standard Universal Algebraic realm.

\subsection{Universal set functors} \label{subsec:res}
The original motivation for studying the universality of set functors was another open problem formulated in~\cite{BT1}: Is the category of varieties and interpretations~\cite{GT84} alg--universal? This category can be described, in the language of previous paragraphs, as the category of sets of arbitrary finitary equations, not necessarily height one. From this perspective, starting with height one equations suggests itself as a reasonable step. The first paper in this direction~\cite{BZ1} has shown, using a rather involved construction, that the category of set functors is group--universal. Then, a significantly simpler construction was carried out in~\cite{Ba2} to prove that the category of finitary set functors is alg--universal. Since this category (as well as the category of $\kappa$--accessible functors for any fixed $\kappa$) is also algebraic, the alg--universality of finitary (or $\kappa$--accessible for $\kappa \geq 7$) set functors cannot be further strengthened. Finally, the motivating problem was solved as well: the category of varieties is alg--universal~\cite{Bar07}.

This development has left the following natural question open: Is the upper bound on the arity of function symbols essential? In other words, is the category of set functors more comprehensive than the category of finitary  ones (equivalently, $\kappa$-accessible for a fixed $\kappa \geq 7$)?
We answer this question in affirmative by showing, in Section~\ref{sec:univ}, that the category of accessible set functors is universal. 
In Section~\ref{sec:concrete} we verify that this result cannot be further strengthened by proving 
that the category of accessible set functors is concretizable. Altogether, we get the following theorem.

\begin{theorem} \label{thm:main}
A category $\cat{Z}$ is isomorphic to a full subcategory of the category of accessible set functors if and only if $\cat{Z}$ is concretizable.
\end{theorem}

\section{Preliminaries} \label{sec:prelim}

\subsection{Notation}
  The set $\{1,2, \dots, n\}$ is denoted $[n]$. We use the notation $A \sqcup B$ ($f \sqcup g$) for a disjoint union of sets (mappings).
	
  Let $f: X \to Y$ be a mapping. The image of $R \subseteq X$ under $f$ is denoted by $f[R]$ and
  the preimage of $S \subseteq Y$ under $f$ is denoted by $f^{-1}[S]$. 
	Mappings are composed from right to left, that is, $f \circ g(x) = fg(x) = f(g(x))$. 
	The identity mapping $X \to X$ is denoted $\id_X$.

\subsection{Categorical concepts}
Let $\cat{M}$ and $\cat{N}$ be categories.
A functor $\Phi: \cat{M} \to \cat{N}$
is \emph{faithful} (\emph{full}, respectively) if, for any two $\cat{M}$--objects $A$ and $B$, it maps
the set of all $\cat{M}$--morphisms from $A$ to $B$ injectively (surjectively, respectively) into the set of all $\cat{N}$--morphisms from $\Phi A$ to $\Phi B$. A \emph{full embedding} is a full and faithful functor which is, moreover, injective on objects. 

A category $\cat{M}$ is a full subcategory of $\cat{N}$ if $\cat{M}$ is obtained from $\cat{N}$ by taking some of the objects and all the morphisms between them. Note that if $\Phi: \cat{M} \to \cat{N}$ is a full embedding, then the image of $\Phi$ is a full subcategory of $\cat{N}$ which is isomorphic to $\cat{M}$. 

The category of all sets and mappings is denoted $\cat{Set}$. A category $\cat{M}$ is \emph{concretizable} if there exists a faithful functor $\cat{M} \to \cat{Set}$. A category $\cat{M}$ is \emph{universal} if every concretizable category has a full embedding into $\cat{M}$. 

A \emph{set functor} is a functor $\cat{Set} \to \cat{Set}$. For set functors $F$ and $G$, and a set $X$, the $X$-th component of a natural transformation $\mu: F \to G$ is denoted $\mu_X: F X \to G X$.

\subsection{Accessible set functors}

A set functor $F$ is called $\kappa$-accessible if every element of $FY$ can be accessed from an element of $FX$ with $|X| < \kappa$ in the following sense. 

\begin{definition}
Let $F: \cat{Set} \to \cat{Set}$ be a set functor and $\kappa$ a cardinal.
$F$ is \emph{$\kappa$-accessible} if 
\[
FY = \bigcup \{Ff[FX] \st |X| < \kappa, f: X \to Y\}
\]
for every set $X$.

A set functor $F$ is \emph{accessible} if it is $\kappa$-accessible for some $\kappa$, and $F$ is \emph{finitary} if it is $\aleph_0$-accessible.
\end{definition}

This definition agrees with the general notion of an accessible functor, see~\cite{AP01} for several equivalent characterizations. 

Strictly formally, accessible set functors do not form a legitimate category since each functor is a proper class. However, every accessible functor has a set presentation which describes the functor up to natural isomorphism. One option is to use an equational presentation as described in~\cite{AGT10}. Alternatively, it can be observed that a $\kappa$-accessible functor (as well as a natural transformation between such functors) is determined by its restriction to the set of cardinals smaller than $\kappa$. 
In this sense, it is legitimate to talk about the category of accessible functors and we denote this category by $\ASF$.

\subsection{Topological spaces without axioms}
If some known ``testing'' universal category $\cat{M}$ can be fully embedded into $\cat{N}$, then $\cat{N}$ is clearly universal as well.
A particularly useful testing category  proved to be ''topological spaces without axioms''. Its objects are pairs consisting of a set and a family of its ``open'' subsets, where, in contrast to topological spaces, we do not impose any restriction on the family of open sets. Morphisms, the ``continuous'' maps, are defined just like in the category of topological spaces.

\begin{definition}
Objects of the category $\cat{T}$ are pairs $(A, \family{R})$, where $A$ is a set and
$\family{R}$ is a family of subsets of $A$.
A morphism $f: (A, \family{R}) \to (B, \family{S})$ is a mapping $f: A \to B$ such that $f^{-1}[S] \in \family{R}$
for all $S \in \family{S}$.
\end{definition}

The category $\cat{T}$ is universal (Hedrl\'in, Ku\v cera, see \cite{PT}).
However, we will construct a \emph{contravariant functor} from $\cat{T}$ 
to $\ASF$, that is, a functor from the opposite category $\cat{T}^{op}$.  
Fortunately, it is not hard to see that the opposite category to a universal category is universal~(see Section 4.6 of~\cite{T5}).

\begin{theorem} \label{thm:topouniv}
The category $\cat{T}$, as well as the opposite category $\cat{T}^{op}$, is universal.
\end{theorem}

\section{Set functors are universal} \label{sec:univ}

In this section we show that the category of accessible functors is universal by constructing
a contravariant functor from $\cat{T}$ to $\ASF$. 
Let us first informally sketch the construction. 

An object $(A, \family{R})$ will be sent to a set functor described by a set of height one equations in a single operation symbol $t$ of arity $|A|$. There will be one equation for each $R$:
$$
t(\underbrace{x, x, \dots, x}_{R}, \underbrace{y, y, \dots, y}_{A \setminus R}) = t(\underbrace{y,y, \dots, y}_R, \underbrace{x, x, \dots, x}_{A \setminus R})\enspace.
$$
This object map naturally extends to a contravariant functor and it is not hard to observe that  $f: (A, \family{R}) \to (B, \family{S})$ is ``continuous'' if and only if the image of $f$ preserves height one equations. 

The only substantial catch in this construction is that the above equation entails the equation where the roles of $R$ and $A \setminus R$ are swapped. For this reason we first construct an auxiliary embedding which will resolve this issue.

\subsection{Auxiliary full embedding}

Our aim now is to construct a full embedding $\Psi: \cat{T} \to \cat{T}$ whose image, denoted $\cat{K}$, has the following properties.
\begin{description}
\item[(K1)] For every $\cat{K}$--object $(A,\family{R})$ and $R \in \family{R}$, we have 
      $A \setminus R \in \family{R}$ and $R \neq \emptyset$.
\item[(K2)] If $f: (A, \family{R}) \to (B, \family{S})$ is a $\cat{K}$-morphism,
      then $|f[A]| > 2$.
\end{description}
  
The construction requires a graph (symmetric, loopless, without multiply edges)	on a sufficiently large even number of vertices with no nonidentical automorphisms. 
Let us fix one such a graph with vertex set  
   $V = \{v_1, v_2, \ldots, v_6\}$ and edge set $\family{G}$, where $\family{G}$ is formally handled as a family of two-element subsets of $V$. The functor $\Psi$ is defined as follows.  
  
  \begin{align*}
  \Psi(A,\family{R}) &= (\overline{A}, \overline{\family{R}})  \\
  {\overline{A}} &= A \sqcup V \\ 
  \overline{\family{R}} &= 
    \{ \{v_i\} \st i \in [6] \} \cup \{\overline{A} \setminus \{v_i\} \st i \in [6]\} \cup \\
 &  \quad
    \{ G \st G \in \family{G}\} \cup \{ \overline{A} \setminus G \st G \in \family{G}\} \cup \\
 &  \quad 
    \{ \{v_1, v_2, v_3\} \cup R \st R \in \family{R}\} \} \cup \{\{v_4, v_5, v_6\} \cup (A \setminus R) \st R \in \family{R}\} \\
  \Psi f  &= f \sqcup \id_V
  \end{align*}
 
  We first check that $\Psi f$ is always continuous. 
	
  \begin{lemma}
	If $f: (A,\family{R}) \to (B, \family{S})$ is a $\cat{T}$--morphism, then 
	so is $\Psi f: (\overline{A},\overline{\family{R}}) \to (\overline{B},\overline{\family{S}})$
  \end{lemma}
	
  \begin{proof}
  Pick $Q \in \overline{\family{S}}$. To prove that $(\Psi f)^{-1}[Q] \in \overline{\family{R}}$ we analyze the six cases in the definition of $\overline{\family{S}}$. 
  \begin{itemize}
    \item If $Q = \{v_i\}$, $i \in [6]$, then $(\Psi f)^{-1}[\{v_i\}] = (f \sqcup \id_V)^{-1}[\{v_i\}] = \{v_i\} \in \overline{\family{R}}$.
    \item If $Q \in \family{G}$, then $(\Psi f)^{-1}[Q] = Q \in \overline{\family{R}}$.
    \item If $Q = \{v_1, v_2, v_3\} \cup S$ with $S \in \family{S}$, then $(\Psi f)^{-1}[Q] = 
      \{v_1, v_2, v_3\} \cup f^{-1}[S] \in \overline{\family{R}}$ since $f$ is a $\cat{T}$--morphism.
  \end{itemize}
  The remaining three cases are complementary and completely analogous. 
  \end{proof}
  
	Clearly, $\cat{K}$ satisfies (K1) and $\Psi$ is a functor (preserves composition and identities) which is faithful and injective on objects.
	It remains to prove that it is full (then (K2) follows as well).
	
	Let
  $(A,\family{R})$ and $(B, \family{S})$ be $\cat{T}$--objects and
  $g: (\overline{A},\overline{\family{R}}) \to (\overline{B},\overline{\family{S}})$ be
  a $\cat{T}$--morphism. We need to show that $g = \Psi f$ for a $\cat{T}$--morphism 
	$f: (A,\family{R}) \to (B, \family{S})$
  
  \begin{lemma} \label{lem:one}
  The function $g$ restricted to $V$ is a bijection $V \to V$ and $g[A] \subseteq B$. 
  \end{lemma}
	
  \begin{proof}
  Let $i \in [6]$ be arbitrary.
  Since $\{v_i\} \in \overline{\family{S}}$, we must have
  $g^{-1}[\{v_i\}] \in \overline{\family{R}}$. By definition of $\overline{\family{R}}$, all the members of $\overline{\family{R}}$ have a nonempty intersection with $V$, therefore
  $g^{-1}[\{v_i\}] \cap V$ is nonempty. We conclude that each $v_i$ has a $g$--preimage in $V$; in other words, $g$ maps $V$ onto $V$.
	Since $V$ is finite, the first claim follows. 
	
	Now we know that, for each $i \in [6]$, the set $g^{-1}[\{v_i\}] \in \overline{\family{R}}$ has a one-element intersection with $V$. The
	only such elements of $\overline{\family{R}}$ are of the form $\{v_j\}$, therefore $g^{-1}[V] \subseteq V$ which is equivalent to the second claim.
	  \end{proof}

  By Lemma~\ref{lem:one}, we can write $g$ as a disjoint union 
  $g = f \sqcup h$, where $f: A \to B$, and $h: V \to V$ is a bijection.
	The following lemma finishes the proof that $\Psi f = g$.
  
  \begin{lemma} \label{lem:three}
  We have $h = \id_V$ and $f : (A, \family{R}) \to (B, \family{S})$ is a $\cat{T}$--morphism.
  \end{lemma}

  \begin{proof}
  For every $G \in \family{G} \subseteq \overline{\family{S}}$, we have $g^{-1}[G] = h^{-1}[G] \in \overline{\family{R}}$. 
  Since $h$ is a bijection, the set $h^{-1}[G]$ has two elements.
  The only such sets in $\overline{\family{R}}$ are the members $\family{G}$, 
	therefore  $h^{-1}[G] \in \family{G}$.  We conclude that $h^{-1}$ is a bijective endomorphism of our fixed
	6-vertex graph. Since this graph is finite, every bijective endomorphism is an automorphism, hence 
	$h^{-1} = \id_V$ and the first part follows.
		
  For every $S \in \family{S}$, we have $\{v_1, v_2, v_3\} \cup S \in \overline{\family{S}}$, therefore (by the previous paragraph)
  $g^{-1}[\{v_1, v_2, v_3\} \cup S] = \{v_1, \dots, v_3\} \cup f^{-1}[S] \in \overline{\family{R}}$ and then $f^{-1}[S] \in \family{R}$, proving the second part.
  \end{proof}

\subsection{Key full embedding}

In this subsection we construct a contravariant full embedding $\Phi$ from the category $\cat{K}$ (see the previous subsection) into 
the category $\ASF$. 

The set functors in the image of $\Phi$ will be quotients of (covariant) hom--functors.
For a $\cat{K}$--object $\alg{A} = (A,\family{R})$ we set
\[
(\Phi \alg{A}) X  = \{g \st g: A \to X \} \modA \mbox{ for any set $X$} \enspace,
\]
where the equivalence $\approx_{\alg{A}}$ is given by
$$
g_1: A \to X \ \approx_{\alg{A}} \  g_2: A \to X \quad\mbox{ if }
$$
\begin{enumerate}
\item $g_1 = g_2$, or 
\item $g_1[A] = g_2[A]$ is a two--element set $\{x,x'\}$ and $g_1^{-1}[\{x\}] = g_2^{-1}[\{x'\}] \in \family{R}$.
\end{enumerate}
The definition of $\approx_{\alg{A}}$ makes sense, since $g_1^{-1}[\{x\}] = g_2^{-1}[\{x'\}] \in \family{R}$ if and only if
$g_1^{-1}[\{x'\}] = g_2^{-1}[\{x\}] \in \family{R}$ as follows from the property (K1). For the same reason, $\approx_{\alg{A}}$ is an equivalence relation. 

The behavior of $\Phi \alg{A}$ on a mapping $h: X \to Y$ is given by
\[
((\Phi \alg{A}) h) (g\modA) = hg\modA \mbox{ for any $g: A \to X$}\enspace.
\]
Clearly, $hg\modA$ does not depend on the choice of representative $g \in g\modA$.

The set functor $\Phi \alg{A}$ is $\kappa$--accessible for any $\kappa > |A|$ since, for each
set $X$ and $g\modA \in (\Phi \alg{A})X$, we have $g\modA = ((\Phi \alg{A})g)(\id_A\modA)$.

It remains to define $\Phi$ on $\cat{K}$--morphisms.
Let  $\alg{A} = (A,\family{R})$ and $\alg{B} = (B,\family{S})$ be $\cat{K}$--objects, and 
 $f: \alg{A} \to \alg{B}$ a $\cat{K}$--morphism. The $X$-th component of the natural transformation
$\Phi f: \Phi \alg{B} \to \Phi \alg{A}$ is defined by the formula
\[
(\Phi f)_X (g\modB) = gf\modA \mbox{ for any $g: B \to X$} \enspace.
\]



In order to see that the definition does no depend on the choice of $g \in g\modB$,
let us consider two $\approx_{\alg{B}}$-equivalent $g_1,g_2: B \to X$. Either $g_1=g_2$, in which case $g_1f= g_2f$,
or $g_1[B] = g_2[B] = \{x,x'\}$, $x \neq x'$, $g_1^{-1}[\{x\}] = g_2^{-1}[\{x'\}] \in \family{S}$.
Since $f$ is a $\cat{K}$--morphism, then $f^{-1}[g_1^{-1}[\{x\}]] = (g_1f)^{-1}[\{x\}]$ and $f^{-1}[g_2^{-1}[\{x'\}]]=(g_2f)^{-1}[\{x'\}]$  are in $\family{R}$. Moreover, neither of these two subsets of $A$ is empty or equal to $A$ by the property (K1), 
therefore $g_1f[A] = g_2f[A] = \{x,x'\}$, hence $g_1f$ and $g_2f$ are $\approx_{\alg{A}}$-equivalent.

\begin{lemma}
For any $\cat{K}$--morphism $f: \alg{A} \to \alg{B}$, the collection $\Phi f$ is a natural transformation $\Phi \alg B \to \Phi \alg A$.
\end{lemma}
\begin{proof}
For any sets $X, Y$, mapping $h: X \to Y$, and $g\modB \in (\Phi \alg{B})X$
we have
$$
( (\Phi\alg{A}) h \circ (\Phi f)_X)(g\modB) = 
((\Phi\alg{A}) h) (gf\modA) = 
[hgf]\modA
$$
and
$$
((\Phi f)_Y \circ (\Phi\alg{B})h )(g\modB) =
(\Phi f)_Y(hg\modB) =
[hgf]\modA \enspace.
$$
Therefore $(\Phi\alg{A}) h \circ (\Phi f)_X = (\Phi f)_Y \circ (\Phi\alg{B})h$.
\end{proof}

The functor $\Phi$ is injective on objects, so it remains to check that $\Phi$ is faithful and full.

\begin{lemma}
$\Phi$ is faithful.
\end{lemma}
\begin{proof}
Let $\alg{A} = (A, \family{R})$, $\alg{B} = (B, \family{S})$ be $\cat{K}$--objects and 
$f_1, f_2: \alg{A} \to \alg{B}$ different morphisms. 
Due to the property (K2), the ranges of $f_1$ and $f_2$ both contain at least three elements,
so $f_1$ and $f_2$ are not $\approx_{\alg{A}}$--equivalent. 
But then $\Phi f_1$ and $\Phi f_2$ differ on $\id_B\modB$ since
$(\Phi f_1)_B([\id_B]_{\alg{B}}) = f_1\modA$ and
$(\Phi f_2)_B([\id_B]_{\alg{B}}) = f_2\modA$. 
\end{proof}

\begin{lemma}
$\Phi$ is full.
\end{lemma}
\begin{proof}
Let $\alg{A} = (A, \family{R}),$ $\alg{B}=(B, \family{S})$ be $\cat{K}$--objects and 
$\mu: \Phi \alg{B} \to \Phi \alg{A}$ a natural transformation. We must check that $\mu = \Phi f$ for some
morphism $f: \alg{A} \to \alg{B}$. 

Let $f: A \to B$ be any mapping such that $f\modA = \mu_B(\id_B\modB)$.
From the naturality of $\mu$ we get that, for any mapping $g: B \to X$,
\begin{align*}
\mu_X (g\modB) &=
(\mu_X \circ (\Phi \alg{B})g)(\id_B\modB) =
 ((\Phi \alg{A})g \circ \mu_B)(\id_B\modB) =
  ((\Phi \alg{A})g)(f\modA) \\ &=
  gf\modB\enspace.
\end{align*}
Now it suffices to show that the mapping $f: A \to B$ is a $\cat{K}$--morphism $\alg{A} \to \alg{B}$,
since the last equality then tells us $\mu = \Phi f$.
Take an arbitrary $S \in \family{S}$ and consider the mappings $g_1,g_2: B \to [2]$ such that
$g_1[S] = \{1\} = g_2[B \setminus S]$ and $g_1[B \setminus S] = \{2\} = g_2[S]$. Clearly $g_1 \approx_{\alg{B}} g_2$, therefore
$\mu_2(g_1\modB) = \mu_2(g_2\modB)$. Thus, using the above equality again, 
$g_1f \approx_{\alg{B}} g_2f$. Since $g_1f \neq g_2f$ (as $g_1(b) \neq g_2(b)$ for any $b \in B$), we get
that $\family{R} \ni (g_1f)^{-1}[\{1\}] = f^{-1}[g_1^{-1}[\{1\}]] = f^{-1}[S]$ and we are done. 
\end{proof}

We have proved one implication in Theorem~\ref{thm:main}.

\begin{theorem}
The category $\ASF$ is universal.
\end{theorem}
\begin{proof}
We have constructed two full embeddings, a covariant $\Psi: \cat{T} \to \cat{K}$ and contravariant $\Phi: \cat{K} \to \ASF$. 
The composition $\Phi\Psi$ can be regarded as a (covariant) full embedding $\cat{T}^{op} \to \ASF$.
But $\cat{T}^{op}$ is universal by Theorem~\ref{thm:topouniv}, so $\ASF$ is universal too. 
\end{proof}

\section{Accessible set functors are concretizable} \label{sec:concrete}

In this section we prove the second implication in Theorem~\ref{thm:main} by showing that $\ASF$ is a concretizable category. We will use Isbell's condition for concretizability, which we now formulate.

Let $\cat{W}$ be a category and $F, G$ its objects. Let $P(F,G)$ denote the class of all pairs of $\cat{W}$--morphisms of the form $(\mu: H \to F, \nu: H \to G)$. We define an equivalence relation $\sim_{F,G}$ on $P(F,G)$ by putting $(\mu, \nu) \sim (\mu',\nu')$ if, for every pair of morphisms $(\alpha: F \to K, \beta: G \to K)$, $\alpha\mu = \beta\nu$ if and only if $\alpha\mu' = \beta\nu'$.

Isbell's condition states that the equivalence relation $\sim_{F,G}$ on $P(F,G)$ has a transversal which is a set.
This condition is necessary for concretizability~\cite{Is63} and, as it turned out, is also sufficent~\cite{Fr73} (a simpler and more constructive proof is given in~\cite{Vin76}).

\begin{theorem} \label{thm:concret}
A category $\cat{W}$ is concretizable if and only if, for any two $\cat{W}$--objects $F,G$, there exists a \emph{set} $Q \subseteq P(F,G)$ containing for every pair $(\mu,\nu) \in P(F,G)$ an $\sim_{F,G}$--equivalent one. 
\end{theorem}

We verify Isbell's condition for the category $\ASF$.
Fix set functors $F$ and $G$ and choose $X$ so that $F$ and $G$ are both $|X|$--accessible and $X$ is infinite. 

For a pair of natural transformations $(\mu: H \to F, \nu: H \to G)$  we consider
the following subsets of $FX \times GX$ and $F\emptyset \times G\emptyset$.
\begin{align*}
T(\mu,\nu) &= \{(\mu_X(x),\nu_X(x)) \in FX \times GX \st x \in HX\} \\
T_0(\mu,\nu) &= \{(\mu_X(x),\nu_X(x)) \in F\emptyset \times G\emptyset \st x \in H\emptyset\} 
\end{align*}
We claim that, for every pair of natural transformations $(\alpha: F \to K, \beta: G \to K)$, we have $\alpha \mu = \beta \nu$ if and only if $\alpha_X(x_1) = \beta_X(x_2)$ for every $(x_1,x_2) \in T(\mu,\nu)$ and $\alpha_{\emptyset}(x_1) = \beta_{\emptyset}(x_2)$ for every $(x_1,x_2) \in T_0(\mu,\nu)$. 
Before proving the claim, observe that Isbell's condition is a consequence. Indeed, for each pair of subsets $R \subseteq FX \times GX$, $R_0 \subseteq F\emptyset \times G\emptyset$ we add to $Q$ a single pair $(\mu,\nu) \in P(F,G)$ with $T(\mu,\nu)=R$ and $T_0(\mu,\nu)=R_0$ provided such a pair exists. 
Since the claim implies that $(\mu,\nu)  \sim_{F,G} (\mu',\nu')$ whenever $T(\mu,\nu)=T(\mu',\nu')$, such a set $Q$ satisfies the requirement in Theorem~\ref{thm:concret}.

One direction of the claim is obvious: if $\alpha \mu = \beta \nu$, then $\alpha_X(x_1) = \beta_X(x_2)$ for every $(x_1,x_2) \in T(\mu,\nu)$ and $\alpha_{\emptyset}(x_1) = \beta_{\emptyset}(x_2)$ for every $(x_1,x_2) \in T_0(\mu,\nu)$. Assume, conversely, that 
\begin{align*}
(\star) \quad &\alpha_X(x_1) = \beta_X(x_2) \mbox{ for every } (x_1,x_2) \in T(\mu,\nu) \enspace \\
  & \alpha_{\emptyset}(x_1) = \beta_{\emptyset}(x_2) \mbox{ for every } (x_1,x_2) \in T_0(\mu,\nu) \enspace.
\end{align*}
Take an arbitrary set $Y$ and an element $y \in HY$. Our aim is to prove that $\alpha_Y \mu_Y (y) = \beta_Y \nu_Y (y)$.

We distinguish two cases according to the cardinality of $Y$. The following lemma deals with the simpler case.

\begin{lemma}
If $|Y| \leq |X|$, then $\alpha_Y \mu_Y (y) = \beta_Y \nu_Y (y)$.
\end{lemma}

\begin{proof}
If $Y=\emptyset$, then the pair $(\mu_{\emptyset} (Hi(y)), \nu_{\emptyset}( Hi(y)))$ is in $T_0(\mu,\nu)$, so $\alpha_{\emptyset}\mu_{\emptyset} Hi (y) = \beta_{\emptyset} \nu_{\emptyset} Hi (y)$ by the second part of $(\star)$.%
\footnote{This is the only place where $T_0(\mu,\nu)$ is used. The case $Y = \emptyset$ was neglected in the previous version of the paper and I thank the reviewer for pointing out this mistake.}

Otherwise, take any injective mapping $i: Y \to X$ and its left inverse $k: X \to Y$ (so that $ki = \id_Y$).
The pair $(\mu_X (Hi(y)), \nu_X( Hi(y)))$ is in $T(\mu,\nu)$, so $\alpha_X\mu_X Hi (y) = \beta_X \nu_X Hi (y)$.
Since $\alpha\mu$ and $\beta\nu$ are natural transformation, we get 
\[
Ki \circ \alpha_Y\mu_Y (y) = \alpha_X\mu_X \circ Hi (y) = \beta_X \nu_X \circ Hi (y) = Ki \circ \beta_Y \nu_Y (y)\enspace.
\]
Then also
\[
Kk \circ Ki \circ \alpha_Y\mu_Y (y) = Kk \circ Ki \circ \beta_Y \nu_Y (y)\enspace,
\]
but $Kk \circ Ki = K(ki) = K\id_{Y} = \id_{KY}$ and the claim follows.
\end{proof}

In the second case, when $|Y| \geq |X|$, we obtain the following consequence of $|X|$--accessibility of $F$ and $G$.

\begin{lemma}
There exist $s \in FX$, $t \in GX$ and injective mappings $i,j: X \to Y$ such that $Fi(s) = \mu_Y(y)$ and $Gj(t) = \nu_Y(y)$.
\end{lemma}

\begin{proof}
As $F$ is $|X|$--accessible, there exists $X'$, $s' \in FX'$, and $f: X' \to Y$ such that $|X'| < |X|$ and $Ff(s') = \mu_Y(y)$.
We factorize $f$ as $f = if'$ with $f': X' \to X$ and injective $i: X \to Y$, we set $s = Ff'(s')$, and get $Fi(s) = Fif'(s') = Ff(s') = \mu_Y(y)$. Finding $t$ and $j$ is completely analogous. 
\end{proof}

We fix $s,t,i,j$ satisfying the conclusion of the previous lemma and continue by finding further mappings that will help us in caluculations.  

\begin{lemma} \label{lemma:aa}
There exist $k,l,m: Y \to X$ and $n: X \to Y$ such that $ki = lj = \id_X$, $nmi=i$, and $nmj = j$.
\end{lemma}

\begin{proof}
The mappings $k$ and $l$ are arbitrarily chosen left inverses to $i$ and $j$, respectively.

As $X$ is infinite, the union $U =i[X] \cup j[X]$ has cardinality $|X|$,  therefore there exists $m: Y \to X$ which is injective on $U$.
Now we take any $n: X \to Y$ such that the restriction of $nm$ to $U$ is the identity
and obtain $nmi=i$ and $nmj=j$, as required.
\end{proof}

Using the mappings from Lemma~\ref{lemma:aa} we get
$$
\alpha_Y\mu_Y (y) =
\alpha_Y Fi (s) = 
\alpha_Y F(iki) (s) =
\alpha_Y F(ik) Fi(s) =
\alpha_Y F(ik) \mu_Y(y)\enspace.
$$
Using this equality, the naturality of $\alpha$ and $\mu$, and Lemma~\ref{lemma:aa} we can continue this calculation as follows.
\begin{align*}
\alpha_Y F(ik) \mu_Y(y) &=
\alpha_Y F(nmik) \mu_Y(y) = 
\alpha_Y F(nm) F(ik) \mu_Y(y) \\
&=
K(nm) \alpha_Y F(ik) \mu_Y (y) =
K(nm) \alpha_Y \mu_Y (y) =
Kn Km \circ \alpha_Y \mu_Y (y) \\
&=
Kn \circ \alpha_X \mu_X Hm (y)\enspace,
\end{align*}
Similarly,
$$
\beta_Y\nu_Y (y) =
Kn \circ \beta_X \nu_X Hm (y)
$$
But $(\mu_X Hm (y),\nu_X Hm (y))$ is in $T(\mu,\nu)$, so $\alpha_X \mu_X Hm (y) = \beta_X \nu_X Hm (y)$ by the assumption $(\star)$ and 
we get $\alpha_Y\mu_Y (y) = \beta_Y\nu_Y (y)$, finishing the proof of the following theorem.

\begin{theorem}
The category $\ASF$ is concretizable.
\end{theorem}

\section{Conclusion}

We have shown that the category of accessible set functors is as comprehensive as possible: it contains all concretizable categories as full subcategories. The next obvious question is how comprehensive is the collection of all set functors. Recall that this collection is not even a class and it is not clear what is a natural limitation on possible full subcollections.

Perhaps a more interesting question concerns the so called \emph{hyper--universality}. A combination of theorems by Ku\v cera~\cite{K4} and Trnkov\'a~\cite{T7} (see Section 4.7 in \cite{T5}) implies that in every universal category it is possible to find equivalences on hom--sets (which are compatible with compositions) so that the quotient category is hyper--universal, that is, it contains \emph{every} category as a full subcategory. Is there a natural hyper--universal quotient of the category of accessible set functors? Is it some kind of a homotopy equivalence on natural transformations?

Finally, let us return to Universal Algebra and finitary set functors. The main result of~\cite{Ba2} says, in the language of~\cite{BKO}, that the category of minions (with possibly infinite universes) is alg--universal. 
However, in the promise CSP we are so far mostly interested in minions with finite universes. How comprehensive is their category?

\bibliographystyle{plain}


\end{document}